\documentclass[12pt,reqno]{amsart}
\usepackage[utf8]{inputenc}
\usepackage{a4wide} % ] {geometry}
\usepackage{amsthm,mathrsfs,amssymb,amsmath,amsfonts,bbm}
\usepackage[hidelinks]{hyperref}
\usepackage{esint, epsfig, babel}
\usepackage{hyperref}
\usepackage{enumerate}
\usepackage{versions}

\usepackage{color}

\allowdisplaybreaks

\numberwithin{equation}{section}

\newcommand{\Hm}[1]{\leavevmode{\marginpar{\tiny%
$\hbox to 0mm{\hspace*{-0.5mm}$\leftarrow$\hss}%
\vcenter{\vrule depth 0.1mm height 0.1mm width \the\marginparwidth}%
\hbox to 0mm{\hss$\rightarrow$\hspace*{-0.5mm}}$\\\relax\raggedright
#1}}}

\newtheorem{theorem}{Theorem}[section]
\newtheorem*{theorem*}{Theorem}
\newtheorem{corollary}[theorem]{Corollary}
\newtheorem{lemma}[theorem]{Lemma}
\newtheorem{proposition}[theorem]{Proposition}
\theoremstyle{definition}

\theoremstyle{remark}

\newtheorem{example}[theorem]{Example}

  \DeclareMathOperator {\supp} {supp\,}
  \renewcommand {\Re}{\text{Re}}
  
  \newcommand {\sfrac}[2] { {{}^{#1}\!\!/\!{}_{#2}}} 
  
  \newcommand {\einhalb}{\sfrac{1}2}

  \newcommand {\CC}{\mathbb C}  
  \newcommand {\EE}{\mathbb E}
  \newcommand {\NN}{\mathbb N}
  
  \newcommand {\ZZ}{\mathbb Z}
  \newcommand {\RR}{\mathbb R}

  \newcommand {\eins} {\mathbbm 1}

  \newcommand {\LL}{{\mathcal L}}

  \newcommand {\om}{\omega}

  \newcommand {\norm}[1] {\| #1 \|}  %% be *very* careful if you whant
  \newcommand {\bignorm}[1]{\bigl\| #1 \bigr\|}
  \newcommand {\Bignorm}[1]{\Bigl\| #1 \Bigr\|}

  \newcommand {\dsp} {\displaystyle}
  
  \newcommand{\MS}{{X}}  %% "metric space"

  \newcommand{\fa}{{\mathfrak a}}
   \newcommand{\fb}{{\mathfrak b}}
   \newcommand{\mA}{\mathcal A}

\newcounter{aufzi}
\newenvironment{aufzi}{\begin{list}{ {\upshape(\alph{aufzi})}}{
        \usecounter{aufzi}
        \topsep1ex
        \parsep0cm
        \itemsep1ex
        \leftmargin0.8cm
        \labelwidth0.5cm
        \labelsep0.3cm
}}
{\end{list}}

\newcounter{aufziH}

\newenvironment{aufziH}{
\begin{list}{\upshape(H\arabic{aufziH})}{
    \usecounter{aufziH}
    \topsep1ex
    \parsep0cm
    \itemsep1ex
    \leftmargin0.8cm
    \labelwidth0.5cm
    \labelsep0.3cm
}}{
\end{list}
}

\newcounter{aufzii}

\newcounter{aufziii}

\begin{document}

\title[${\mathcal R}$-boundedness of Riesz transforms]{ ${\mathcal R}$-boundedness of Riesz transforms and spectral multipliers of Schr\"odinger operators} 

\author{El Maati~Ouhabaz} \address{ Institut de Math\'ematiques de Bordeaux, Universit\'e de Bordeaux, UMR CNRS 5251, 351 Cours de la Lib\'eration 33405, Talence.  France.
\newline Current address:  International Research Laboratory FAMSI, CNRS-ANU, Canberra, ACT 2601, Australia.}  
\email{Elmaati.Ouhabaz@math.u-bordeaux.fr}
 
 \begin{abstract}
Let $\LL = -\Delta + V$ be a Schr\"odinger operator with a non-negative potential. It is known that its associated Riesz transform $\nabla \LL^{-\sfrac{1}{2}}$ is bounded on $L^p(\RR^d)$ for $p \in (1,2]$, with a bound independent of $V$. We prove more, namely that the set
\[
\{ \nabla \LL^{-\sfrac{1}{2}},\ 0 \le V \in L^1_{\mathrm{loc}}(\RR^d) \}
\]
is ${\mathcal R}$-bounded on $L^p(\RR^d)$ for $p \in (1,2]$. We also prove the ${\mathcal R}$-boundedness on $L^p(\RR^d)$, for $p \in (1,\infty)$, of the spectral multipliers
\[
{ m(\LL),\ 0 \le V \in L^1_{\mathrm{loc}}(\RR^d) }
\]
for any $m:[0, \infty) \to \CC$ satisfying a H\"ormander type condition. 

Our approach is based on the analysis of Schr\"odinger operators $-{\bf \Delta} + {\bf V}$ on vector-valued spaces $L^p(\RR^d,\ell^2)$. In particular, we prove heat kernel estimates, Riesz transform bounds, and spectral multiplier results.
\end{abstract}
\vspace{1cm}

\maketitle

\noindent{\bf Keywords}:  Riesz transforms,  spectral
multipliers, Schr\"odinger operators, vector-valued singular integrals.

%\vspace{1cm} 
\noindent {\bf Mathematics Subject Classification}: 
47D08, %(Schrod oper)
42B20, %Harmonic Analysis singular integrals,
42B37%Harm analysis and PDEs}

 \vspace{1cm}
 %\tableofcontents

 \section{Introduction and main results}\label{sec1} 
 In recent years, systems of elliptic or parabolic problems involving Schr\"odinger-type operators have attracted considerable attention. Questions concerning well-posedness in $L^p$-spaces, regularity properties, Riesz transforms, functional calculi, and related topics have been studied extensively. We refer the reader to \cite{ALLR, ALLOR, ALLR-2, AFL, ALM, ALMR} and the references therein for an overview.

In these papers, the setting is either that of systems of coupled elliptic or parabolic operators or that of Schr\"odinger operators on $L^p(\RR^d,\CC^m)$. The results often depend on the dimension $m$, which does not allow one to consider the vector-valued setting $L^p(\RR^d,\ell^2)$. In the present paper, we focus mainly on two problems: Riesz transforms and spectral multipliers for vector-valued Schr\"odinger operators $-{\bf \Delta} + {\bf V}$ on $L^p(\RR^d,\ell^2)$. The vector potential ${\bf V}$ is symmetric and non-negative in the sense that
\[
\langle {\bf V}(x)\varphi,\varphi\rangle_{\ell^2} \ge 0
\]
for all $\varphi \in \ell^2$. Specialising to diagonal potentials gives the following results.

%%%%%%%%%%%%
\begin{theorem}\label{thm1.1} The Riesz transforms satisfy the following properties. 
\begin{itemize}
\item There exists a constant $C > 0$ such that for every sequence of non-negative potentials $V_k \in L^1_{loc}(\RR^d)$ we  have 
$$ \norm{\Big(\sum_k | \nabla (-\Delta + V_k)^{-\einhalb} f_k |^2 \Big)^{\frac{1}{2}}}_{L^{1,\infty}(\RR^d)} \le C\, \norm{\Big(\sum_k | f_k |^2\Big)^{\frac{1}{2}}}_{L^1(\RR^d)}.$$
\item Let $p \in (1, 2]$. Then the family $\{ \nabla (-\Delta + V)^{-\einhalb}, \, 0 \le V \in L^1_{loc}(\RR^d) \}$ is ${\mathcal R}$-bounded on $L^p(\RR^d)$. 
\item Let $p \in (1, 2]$ and $q \in [p, 2]$. Then there exists a positive constant $C$ such that for every sequence of nonnegative potentials $V_k \in L^1_{loc}(\RR^d)$ we have
$$ \norm{\Big(\sum_k | \nabla (-\Delta + V_k)^{-\einhalb} f_k |^q \Big)^{\frac{1}{q}}}_{L^p(\RR^d)} \le C\, \norm{\Big(\sum_k | f_k |^q\Big)^{\frac{1}{q}}}_{L^p(\RR^d)}.$$
\end{itemize} 
\end{theorem} 

For spectral multipliers, we have
\begin{theorem}\label{thm1.2}
Let $m: [0, \infty) \to \CC$ be a bounded function which satisfies the following H\"ormander condition
\begin{equation}\label{eq1.1}
\sup_{t>0} \norm{ \psi(\cdot) m(t\cdot)}_{W^{s,\infty}(\RR)} < \infty
\end{equation}
for some $s > d/2$ and some non-trivial $\psi \in C_c^\infty(0,\infty)$. Then
\begin{itemize}
\item There exists a constant $C > 0$ such that for every sequence of non-negative potentials $V_k \in L^1_{loc}(\RR^d)$ we  have 
$$ \norm{\Big(\sum_k |  m(-\Delta + V_k) f_k |^2 \Big)^{\frac{1}{2}}}_{L^{1,\infty}(\RR^d)} \le C\, \norm{\Big(\sum_k | f_k |^2\Big)^{\frac{1}{2}}}_{L^1(\RR^d)}.$$
\item Let $p \in (1, \infty)$. Then the family $\{ m(-\Delta + V), \, 0 \le V \in L^1_{loc}(\RR^d) \}$ is ${\mathcal R}$-bounded on $L^p(\RR^d)$. 
\item Let $p \in (1, \infty)$ and $q \in [\min(p,2), \max(p,2)]$. Then there exists a positive constant $C$ such that for every sequence of nonnegative potentials $V_k \in L^1_{loc}(\RR^d)$ we have
$$ \norm{\Big(\sum_k | m(-\Delta + V_k) f_k |^q \Big)^{\frac{1}{q}}}_{L^p(\RR^d)} \le C\, \norm{\Big(\sum_k | f_k |^q\Big)^{\frac{1}{q}}}_{L^p(\RR^d)}.$$
\end{itemize} 
\end{theorem} 
 
 For the maximal H\"ormander functional calculus we have the following corollary.
 \begin{corollary}\label{cor1.3}
 Suppose that $m$ is supported in $[\frac{1}{2}, 2]$ and satisfies \eqref{eq666} for some $s > \frac{d}{2} +2$. Then there exists a constant $C$ such that for every sequence of non-negative potentials $V_k \in L^1_{loc}(\RR^d)$ we  have 
 \[
 \norm{\Big( \sum_k \sup_{r > 0} |m(r(-\Delta + V_k) ) f_k|^2 \Big)^{\frac{1}{2}}}_{L^p(\RR^d)} \le C\,
 \norm{\Big(\sum_k | f_k |^2\Big)^{\frac{1}{2}}}_{L^p(\RR^d)} 
 \]
for $p \in (1, \infty)$.
 \end{corollary}

 Recall that a family ${\mathcal T} =(T_i)_{i \in I}$ of bounded operators on $L^p$ is ${\mathcal R}$-bounded if there exists a constant $C > 0$ such that for every 
 collection $T_1, \cdots, T_n \in {\mathcal T}$ and $f_1, \cdots, f_n \in L^p$, we have
 \begin{equation}\label{eq0.1}
\EE \norm{\sum_{k=1}^n \varepsilon_k T_k f_k }_p  \le C\, \EE \norm{\sum_{k=1}^n \varepsilon_k f_k}_p
\end{equation}
where $(\varepsilon_k)_k$ is a sequence of independent Rademacher variables and $\EE$ is the expectation.
By the Kahane-Khintchine inequality, ${\mathcal R}$-boundedness can be reformulated as 
\begin{equation}\label{eq0.2}
\norm{ \Big( \sum_{k=1}^n | T_k f_k|^2 \Big)^{1/2}  }_p \le C\,  \norm{ \Big( \sum_{k=1}^n |f_k|^2 \Big)^{1/2} }_p
\end{equation}
again for every collection $T_k \in {\mathcal T}$ and $f_k \in L^p$. It is immediately  clear from this that a natural way to prove the ${\mathcal R}$-boundedness of the Riesz transforms  or the multipliers is  to study the Schr\"odinger operator $-{\bf \Delta} + {\bf V}$ on  $L^p(\RR^d, \ell^2)$. In order to do so we appeal to some singular integral operator techniques on vector-valued spaces. The above two results leave, however, several questions open. One of them is the limitation $q \in [p,2]$ (or $[2,p]$) in the last assertion in each theorem. Whether those assertions hold for arbitrary $q \in [1, \infty]$ is not known to the author of this paper. Another question is whether the constant $C$ can be independent of the dimension $d$. For a single Riesz transform $\nabla (-\Delta + V)^{-\einhalb}$, the weak type $(1,1)$ as well as the $L^p$-norm (for $p \in (1, 2]$) can be estimated by a constant independent of $d$ and $V$. We refer to the recent paper \cite{Dziubanski} for this generality and the references there for previously known results for some specific potentials. 
In Section~\ref{sec2-2} below we reproduce the elegant arguments of \cite{Dziubanski} with some minor changes to make the proof independent of the specific operator $\Delta$ on the Euclidean space.  The fact that we use singular integrals does not allow us to have dimension-free estimates. This needs  a different approach and one could try to take inspiration from \cite{DK}, where dimension free estimates are proved  for the Hardy-Littlewood maximal operator on $L^p(\RR^d, \ell^q)$.  

We finish this introduction by an example which shows $L^p$-estimates for the gradient of solutions to a system of evolution equations. 
 \begin{example}\label{ex1.1}
 Consider the following system of independent $N$ evolution equations with nonnegative potentials $V_k$, 
 \[
 \begin{cases}
\dfrac{\partial u_1}{\partial t} = \Delta u_1 - V_1 u_1, \quad u_1(0) = f_1 \\[1ex]
\dfrac{\partial u_2}{\partial t} = \Delta u_2 - V_2 u_2, \quad u_2(0) = f_2\\
\quad \vdots \\
\dfrac{\partial u_N}{\partial t} = \Delta u_N - V_N u_N, \quad u_N(0) = f_N.
\end{cases}
\]
 Then for $p \in (1, 2]$ there exists a constant $C$, independent of $V_k$ and of the size $N$ of the system, 
 such that
 \[
 \norm{ \left(\sum_{k=1}^N | \nabla u_k |^2\right)^\einhalb}_{L^p(\RR^d)} \le \frac{C}{\sqrt{t}}\, 
 \norm{ \left(\sum_{k=1}^N | f_k|^2\right)^\einhalb}_{L^p(\RR^d)}.
 \]
 In order to see this we write  $u_k = e^{-t(-\Delta + V_k)}f_k$ and 
 \[
 \sqrt{t} \nabla e^{-t(-\Delta + V_k)}f_k = \nabla (-\Delta + V_k)^{-\einhalb} (t(-\Delta + V_k))^\einhalb e^{-t(-\Delta + V_k)}f_k.
 \]
We apply Theorems~\ref{thm1.1} and \ref{thm1.2}. The first one for the Riesz transform and the second for the multiplier with the function 
 $m_t(\lambda) = \sqrt{t\lambda}\, e^{-t \lambda}$. 
 \end{example} 
 
 %\hskip
 \noindent{\bf Notation}.  We use the following classical notation. If $(X, \mu)$ is a $\sigma$-finite measure space, the usual Lebesgue spaces $L^p(X,\mu)$ are often denoted by $L^p(X)$ and their norms by $\norm{\cdot}_p$. For a given Banach space $E$, $L^p(X, \mu, E)$ denotes the Bochner $L^p$ space  of functions with  values in $E$. The norm of $E$ will be denoted by $\norm{\cdot}_E$. The space of bounded linear operators from $E$ to $F$ is denoted by ${\mathcal L}(E,F)$. For $T \in {\mathcal L}(L^p,L^q)$ we use the notation
 $\norm{T}_{p-q}$ for its norm. For a given Hilbert space $H$, we denote by $\langle \cdot, \cdot \rangle_H$ its scalar product. 
 
 \medskip
\noindent{\bf Acknowledgments.} The author wishes to thank Christoph Kriegler for pointed out  the statement in Corollary~\ref{cor1.3} and for several interesting discussions on the subject of this paper.

%%%%%%%%%%%%%%%%%%%%%%%%%%%%%%%%%%%%%
\section{Some facts on vector-valued singular integral operators}\label{sec1.5} 
In this section we review some results on operators given singular integrals acting on vector-valued $L^p$-spaces. The material here is essentially taken from the recent paper \cite{HO}, but we make some observations which will be useful for vector-valued Schr\"odinger operators. 

Let $(\MS, \rho, \mu)$ be a metric measure space and set $
  B(x,r) = \{y \in \MS: \; \rho(x,y) < r \}$ to be  
the open ball centered at $x \in \MS$ with radius $r > 0$. Its volume is
denoted by $V(x,r) = \mu( B(x,r) )$. For $j \ge 1$, the annulus
$B(x, (j+1)r) \setminus B(x,j r)$ is denoted by $C_j(x,r)$ and
$C_0(x, r) := B(x, r)$.  We suppose that $(\MS, d, \mu)$ is of
homogeneous type.  That is, there exists a constant $C > 0$ such that 
%that is, the volume satisfies the doubling property
\[
  V(x, 2r) \le C \, V(x,r) \quad \forall  x \in \MS, \ r >0
\]
for every $x $ and $r$. In this case,  it is elementary that there exist constants  $n > 0$ and  $C_n$ such that
\begin{equation}
  \label{eq:gonflement-lambda}
  V(x, \lambda r) \le C_n \, \lambda^n V(x,r) \quad \forall  x \in \MS, \ r >0 \ {\rm and}\ \lambda \ge 1.
\end{equation}

\medskip
The following result provides a weak type extrapolation principle for operators. 
%A DISCUTER: PRENDRE LES ANNEAUX DYADIC  ??? \\

\begin{theorem}\label{thm:main}(\cite{HO}, Theorem~1.1)
  Let $(\MS, \rho, \mu)$ be a metric measure space of homogeneous type,
  let $(E, \norm{\cdot}_E)$ and $(F, \norm{\cdot}_F)$ be Banach
  spaces and let $p_0, p, q, a \in [1, \infty)$ be such that $p_0 < p \le q$
  where $\frac{1}{p} - \frac{1}{q} = \frac{1}{p_0} -  \frac{1}{a}$. Let 
  $T : L^p(\MS, \mu, E) \to L^q(\MS, \mu, F)$ be a bounded linear
  operator and suppose that there exists a bounded linear operator $S : L^p(\MS, \mu, E) \to 
  L^q(\MS, \mu, F)$
  and a family of linear operators
  $(A_r)_{r>0}$ on $L^p(\MS, \mu, E)$ such that
\begin{aufziH}
\item \label{thm:main-a} for some sequence $(\omega_j)$ of
  non-negative numbers satisfying
  \( \dsp \sum_{j\ge 1} j^n \omega_j < \infty \) and for each
  $f \in L^{p_0}(\MS, E)$ with bounded support and each ball $B(x, r)$
  containing its support, we have the off-diagonal bound
  \begin{equation} \label{eq:At-estimate}
    \Bigl(\tfrac{1}{V(x,(j+1)r)} \int_{C_j(x, r)} \bignorm{ (A_r f)(y) }_E^p  \,d\mu(y)\Bigr)^{\frac{1}{p}} \le \frac{\om_j}{V(x,r)^{\frac{1}{p_0}}}
    \; \norm{ f }_{L^{p_0}(\MS, \mu,E)}.
\end{equation}
\item there exist  $\delta, W > 0$ such that
\begin{equation} \label{eq:Hoermander-a-condition}  
 \Bigl( \int_{\MS \setminus B(x, (1{+}\delta)r)} \bignorm{ (T - SA_r) f(y) }_F^a\,d\mu(y)\Bigr)^{\frac{1}{a}}  \quad \le \quad W  \,\norm{ f  }_{L^{p_0}(\MS, \mu,E)}
\end{equation}
for all $x\in \MS$, $r>0$ and
$f \in L^{p_0}(\MS, \mu, E)\cap L^\infty(\MS, \mu, E)$ supported in
$B(x, r)$.
\end{aufziH}
\medskip\noindent
Then $T: L^{p_0}(\MS, \mu,  E) \to L^{a, \infty}(\MS, \mu, F)$ is bounded.\\
In the special case that $p_0{=}a{=}1$, the operator $T$ is of weak-type $(1,1)$. 
\end{theorem}

Note that the same result holds with $\MS$ replaced  by any non-empty open subset $\Omega$, see \cite{HO} Corollary~1.3. It is worth mentioning  that  condition \eqref{eq:At-estimate} is satisfied in the particular case where $A_r, r > 0$ are given by kernels which enjoy a Gaussian upper bound. 

  Now, suppose, in
  addition to the hypotheses of Theorem~\ref{thm:main} that the
  operators $T$ and $TA_r$ are given by (singular) kernels
  $\vec{K}_T(x,y)$ and $\vec{K}_{TA_r}(x,y)$. That is, $\vec{K}_T$ and
  $\vec{K}_{TA_r}$ are strongly measurable functions on $\MS\times\MS$
  with values in ${\mathcal L}(E, F)$ and locally integrable on
  $\MS\times\MS \setminus \{(x,x): \; x \in \MS\}$ ensuring that
  \begin{equation}\label{eq:2-2-sing}
    T f(x) = \int_\MS \vec{K}_T(x,y) f(y) \, d\mu(y)
  \end{equation} 
  is well defined as an element of $F$ for all
  $f \in L^\infty(\MS, \mu , E)$ that have a bounded support and for (a.e.) 
  $x \not\in \supp(f)$ (and similarly for $TA_r$ and
  $\vec{K}_{TA_r}(x,y)$). Then we have

\begin{proposition}\label{prop:prop2-1}
  Let $a \in [1, \infty)$ and $ \delta > 0$. Consider the following
  properties.
\begin{aufzi}
\item \label{prop2-1-1}(Hörmander condition) There exists a constant $W_0$ such that for all $y,  y' \in \MS$,
\begin{equation}\label{eq:2-1-Ho}
  \int_{\rho(x,y) \ge (1{+}\delta) \rho(y,y')}  \norm{ \vec{K}_T(x,y) u - \vec{K}_{T}(x,y') u }_{F}^a  \, d\mu(x) \le W_0\, \norm{u}_E^a \quad{\rm for \ all}\  u \in E.
  \end{equation}
\item \label{prop2-1-2} There exists a family $(A_r)_{r>0}$ of linear
  operators on $L^p(\MS, \mu, E)$ which satisfies the off-diagonal
  bound \eqref{eq:At-estimate} and a constant $W_1$ such that
\[
  \sup_{y \in \MS, r >0} \; 
  \int_{\rho(x,y) \ge (1{+}\delta)r} \norm{ \vec{K}_T(x,y) u - \vec{K}_{TA_r}(x,y) u }_{F}^a  \, d\mu(x) \le W_1\, \norm{u}_E^a \quad{\rm for \ all}\  u \in E. 
\] 
\item \label{prop2-1-3} There exists  a constant $ W_2 > 0$ such that
\[
 \int_{\MS \setminus B(x, (1{+}\delta)r)} \bignorm{ (T - TA_r)    f(y) }_F^a\,d\mu(y)  \le \quad W_2 \,\norm{
    f }_{L^{1}(\MS,\mu,E)}^a
\]
for all $x\in \MS$, $r>0$ and
$f \in L^{1}(\MS,\mu, E)\cap L^\infty(\MS,\mu, E)$ supported in
$B(x, r)$.
\end{aufzi}
Then
$\ref{prop2-1-1} \Rightarrow \ref{prop2-1-2} \Rightarrow
\ref{prop2-1-3}$. In particular, conditions \ref{prop2-1-1} and  \ref{prop2-1-2} both imply that the operator $T$ is of weak type $(1, a)$ as soon as 
  $ T : L^p(\MS, \mu, E) \to L^q(\MS, \mu, F)$ is bounded for some $p$ and $q$ such that $\frac{1}{p} - \frac{1}{q} = \frac{1}{p_0} -  \frac{1}{a}$. 
\end{proposition} 

Before we discuss the proof we make some comments. The statement  appears in \cite{HO}, Proposition~1.5, but with condition
\begin{equation}\label{eq:2-1-Ho-1}
  \sup_{y,  y' \in \MS} \int_{\rho(x,y) \ge (1{+}\delta) \rho(y,y')}  \norm{ \vec{K}_T(x,y) - \vec{K}_{T}(x,y') }_{{\mathcal L}(E,F)}^a  \, d\mu(x) < \infty
  \end{equation}
  instead of \eqref{eq:2-1-Ho} and 
  \begin{equation}\label{eq:2-1-HO-2}
  \sup_{y \in \MS, r >0} \; 
  \int_{\rho(x,y) \ge (1{+}\delta)r} \norm{ \vec{K}_T(x,y) - \vec{K}_{TA_r}(x,y) }_{{\mathcal L}(E,F)}^a  \, d\mu(x) < \infty 
\end{equation}
in \ref{prop2-1-2}. The aim of the above proposition is to show that the "strong" estimates in \ref{prop2-1-1} and \ref{prop2-1-2},  instead of the "uniform norm" ones in \eqref{eq:2-1-Ho-1}
and \eqref{eq:2-1-HO-2}, are enough to obtain weak type $(1,a)$. This refinement will be very useful in the study of Riesz transforms of vector-valued Schr\"odinger operators. In this setting, the heat kernel  ${\bf p}(t,\cdot,y)$,   
 is a function which takes values in  ${\mathcal L}(\ell^2)$ and it is not an element of the form domain (in contrast to  ${\bf p}(t,\cdot,y)u$ for $u \in \ell^2$). This simple fact plays an important role in the proof of some weighted gradient estimates of the heat kernel. See Section~\ref{sec3}. 
 In this vector-valued setting,  the H\"ormander  condition, with $a=1$, is formulated as in \eqref{eq:2-1-Ho-1} with the uniform norm in  \cite{Grafakos-Liu-Yang} who proved that it implies weak type $(1,1)$. Another reason why it is preferable to work with \ref{prop2-1-1} and \ref{prop2-1-2} comes from measurability  issues for operators in the vector-valued setting. In general it is hard to have measurability in ${\mathcal L}(E,F)$ for the uniform norm so that  one  can  write \eqref{eq:2-1-Ho-1} and \eqref{eq:2-1-HO-2}. On the contrary, measurability in the strong sense is easier and follows from weak measurability for separable spaces via Pettis' theorem.

\begin{proof} We follow closely \cite{HO}, which uses similar ideas as in \cite{Duong-McIntosh}. We prove that
  $\ref{prop2-1-1} \Rightarrow \ref{prop2-1-2}$. Let $u \in E$ and set 
\[
J(x,y) = \norm{ \vec{K}_T(x,y)u - \vec{K}_{TA_r}(x,y)u }_{F}^a  \quad{\rm and} \quad I = \int_{\rho(x,y) \ge (1{+}\delta)r}   J(x,y) \, d\mu(x),
\]
where $A_r$ is the operator defined by the kernel
\[
  \vec{h_r}(z,y) = \frac{\eins_{B(y,r)}(z)}{V(y,r)} I_E
\]
in which $I_E$ denotes the identity operator on $E$. Clearly, this kernel has  a Gaussian upper bound and so  \eqref{eq:At-estimate} holds for $p_0 = 1$. The kernel
$\vec{K}_{TA_r}(x,y)$ is then given by the usual composition
formula. Hence
\begin{align*}
J(x,y) = &  \; \norm{ \int_{\rho(z,y) \le r} [ \vec{K}_T(x,y) - \vec{K}_{T}(x,z)] \vec{h_r}(z,y)u\, d\mu(z) }_{F}^a \\
= &  \; \norm{ \int_{\rho(z,y) \le r} \frac{\eins_{B(y,r)}(z)}{V(y,r)} [ \vec{K}_T(x,y)u - \vec{K}_{T}(x,z)u] \, d\mu(z) }_{F}^a \\
\le& \; \left( \int_{\rho(z,y) \le r} \norm{ \vec{K}_T(x,y)u - \vec{K}_{T}(x,z)u}_{F} \frac{\eins_{B(y,r)}(z)}{V(y,r)} \, d\mu(z) \right)^a\\
\le& \;  \int_{\rho(z,y) \le r} \norm{ \vec{K}_T(x,y) u- \vec{K}_{T}(x,z)u}_{F}^a \eins_{B(y,r)}(z) \, d\mu(z) V(y,r)^{-1}.
\end{align*}
Using Fubini and  \ref{prop2-1-1}, we obtain 
\begin{align*}
I \le& \; \int_{\rho(z,y) \le r} \int_{\rho(x,y) \ge (1+ \delta)d(z,y)}\norm{ \vec{K}_T(x,y)u - \vec{K}_{T}(x,z)u}_{F}^a\, d\mu(x)  \eins_{B(y,r)}(z) \, d\mu(z) V(y,r)^{-1}\\
  \le& \; W_0\,  \norm{u}_E^a  \int_{d(z,y) \le r} \eins_{B(y,r)}(z) \, d\mu(z) V(y,r)^{-1} = W_0\,  \norm{u}_E^a,
\end{align*}
which is  property \ref{prop2-1-2}. 

\smallskip\noindent Assume now that \ref{prop2-1-2} is satisfied. Let
$f \in L^{1}(\MS,\mu, E)\cap L^\infty(\MS,\mu, E)$ with support
contained in a ball $B(x, r)$. Then,
\begin{align*}
    &  \; \Bigl( \int_{\MS \setminus B(x, (1{+}\delta)r)} \bignorm{ (T - TA_r) f(y) }_F^a\,d\mu(y)\Bigr)^{\frac{1}{a}}\\
   =&  \; \Bignorm{ \int_\MS  ( \vec{K}_T(\cdot,z) - \vec{K}_{TA_r}(\cdot,z)) f(z)\, d\mu(z)}_{L^a(\MS \setminus B(x, (1{+}\delta)r), \mu , F)}\\
\le &  \; \int_\MS  \bignorm{ ( \vec{K}_T(\cdot,z)f(z) - \vec{K}_{TA_r}(\cdot,z) f(z))\, d\mu(z)}_{L^a(\MS \setminus B(x, (1{+}\delta)r), \mu ,  F)}\\
=&   \; \int_{\MS} \left( \int_{\MS \setminus B(x, (1{+}\delta)r)} \norm{\vec{K}_T(y,z)f(z) - \vec{K}_{TA_r}(y,z)f(z)}_{F}^a\, d\mu(y) \right)^{\tfrac{1}a} \, d\mu(z)\\
\le&   \; W_1^{\tfrac{1}a}\,   \int_\MS \norm{f(z)}_{E}\, d\mu(z),
\end{align*}
which  proves  property \ref{prop2-1-3}.\\
 The final statement in the proposition follows from Theorem~\ref{thm:main}. 
\end{proof}

%%%%%%%%%%%%%%%%%%%%%%%%%%%%%%%%%%%%%%%%%%%%%%
\section{$L^p$-bounds for Riesz transforms of Schr\"odinger operators}\label{sec2-2}
Very recently, it was proved in \cite{Dziubanski} that for $d \ge 3$ the Riesz transform $\nabla (-\Delta + V)^{-\einhalb}$ is bounded on $L^p(\RR^d)$  for $p \in (1, 2]$ with bound given 
by \eqref{eq2.1} below. The elegant arguments there are  quite elementary.  
In the theorem below, we state the main result from \cite{Dziubanski} and give a proof. The ideas are the same as in that paper but we do not rely on integral kernels or Green functions so that the  assumption $d \ge 3$ can be ignored. In addition,  the proof below does not use explicitly  the expression of the heat kernel and hence it is valid  in a more general setting such as complete Riemannian manifolds or operators with variable coefficients on Euclidean domains. %as soon as one know that the Riesz transform $\nabla (-\Delta)^{-\sfrac{1}{2}}$ is bounded on $L^p$. 

\begin{theorem} \label{thm2.1}(\cite{Dziubanski}). 
Let $d \ge 1$, $ p \in (1, 2]$  and let $\LL = -\Delta + V$ be a Schr\"odinger operator with a non-negative potential $V \in L^1_{loc}(\RR^d)$. Then the Riesz transform $\nabla \LL^{-\sfrac{1}{2}}$ is bounded on $L^p(\RR^d)$ and we have the estimate 
\begin{equation}\label{eq2.1}
\norm{\nabla \LL^{-\sfrac{1}{2}}}_{L^p-L^p} \le 2^{\frac{2-p}{p}} \, \norm{\nabla (-\Delta)^{-\sfrac{1}{2}}}_{L^p-L^p}.
\end{equation}
In particular, the $L^p$-norm of $\nabla \LL^{-\sfrac{1}{2}}$ can be estimated by a constant independent of $d$ and $V$.\\
For $p=1$, one has the following weak type $(1,1)$ estimate
\begin{equation}\label{eq2.1-1}
\norm{\nabla \LL^{-\sfrac{1}{2}}}_{L^1-L^{1,\infty}} \le 2 \, \norm{\nabla (-\Delta)^{-\sfrac{1}{2}}}_{L^1-L^{1,\infty}}.
\end{equation}
\end{theorem} 

\begin{proof} We recall the following basic  formulas
\begin{equation}\label{eq2.2}
\LL^{-\sfrac{1}{2}} = \frac{2}{\pi} \int_0^\infty (1+ t^2 \LL)^{-1}\, dt \quad{\rm and} \quad \LL^{\sfrac{1}{2}} = \LL \LL^{-\sfrac{1}{2}} = \frac{2}{\pi} \int_0^\infty [ I - (1+t^2\LL)^{-1} ] \, \frac{dt}{t^2}.
\end{equation}
One may also replace $\LL$ by $\LL + \epsilon$ to have an invertible operator and then avoid some domain questions in \eqref{eq2.2}. One proves \eqref{eq2.1} for $\LL + \epsilon$ and then let $\epsilon \to 0$. We shall ignore these simple details  in the arguments below.
We also recall the following well known $L^1$-maximal inequalities 
\begin{equation}\label{eq2.3}
\norm{V u}_{L^1} \le  \norm{\LL u}_{L^1} \quad {\rm and } \quad \norm{\Delta u}_{L^1} \le 2\,  \norm{\LL u}_{L^1}.
\end{equation} 
For simplicity of notation, we set $\LL_0 = - \Delta$. Following the ideas in \cite{Dziubanski} we estimate the $L^1$-norm of the operator $\LL_0^{\sfrac{1}{2}} \LL^{-\sfrac{1}{2}}$.  We have 
\begin{eqnarray*}
\LL_0^{\sfrac{1}{2}} \LL^{-\sfrac{1}{2}} f &=& \frac{2}{\pi} \int_0^\infty [ I - (1+t^2\LL_0)^{-1} ] \LL^{-\sfrac{1}{2}} f \, \frac{dt}{t^2}\\
&=& \frac{2}{\pi} \int_0^\infty [  (1+t^2\LL)^{-1} - (1+t^2\LL_0)^{-1} ] \LL^{-\sfrac{1}{2}} f \, \frac{dt}{t^2} +\\
 &&   \frac{2}{\pi} \int_0^\infty [ I - (1+t^2\LL)^{-1} ] \LL^{-\sfrac{1}{2}} f \, \frac{dt}{t^2}\\
&=& f - Tf
\end{eqnarray*}
with 
$$ Tf = \frac{2}{\pi} \int_0^\infty [  (1+t^2\LL_0)^{-1} - (1+t^2\LL)^{-1} ] \LL^{-\sfrac{1}{2}} f \, \frac{dt}{t^2}.$$
We explicit the difference of the two resolvents in the expression of $Tf$ so that 
$$ Tf =  \frac{2}{\pi} \int_0^\infty [  (1+t^2\LL_0)^{-1} V (1+t^2\LL)^{-1} ] \LL^{-\sfrac{1}{2}} f \, dt.$$
Observe that $T$ is  a positive operator since all the resolvents as well as $\LL^{-\sfrac{1}{2}}$ are positive operators. Hence we can restrict our attention in the estimate of $\norm{Tf}_{L^1}$ to nonnegative $f$. Now, for nonnegative $f$,
\begin{eqnarray*}
\norm{Tf}_{L^1} &=& \int_{\RR^d} T f(x) \, dx\\
&=& \frac{2}{\pi} \int_0^\infty \int_{\RR^d}   (1+t^2\LL_0)^{-1} V (1+t^2\LL)^{-1}  \LL^{-\sfrac{1}{2}} f(x)\, dx\, dt\\
&\le& \frac{2}{\pi} \int_0^\infty \int_{\RR^d}   V (1+t^2\LL)^{-1}  \LL^{-\sfrac{1}{2}} f(x)\, dx\, dt\\
&=& \frac{2}{\pi} \int_{\RR^d}   V \LL^{-1} \int_0^\infty (1+t^2\LL)^{-1}  \LL^{\sfrac{1}{2}} f(x)\, dt\, dx\\
&=& \int_{\RR^d} V \LL^{-1} f (x)\, dx.
\end{eqnarray*}
Note that in the third line we used the fact that the resolvent $(1+t^2\LL_0)^{-1}$ is a contraction on $L^1(\RR^d)$. We apply the maximal inequality \eqref{eq2.3} to estimate the latest term by $\norm{f}_{L^1}$. Thus,
\begin{equation}\label{eq2.4}
\norm{\LL_0^{\sfrac{1}{2}} \LL^{-\sfrac{1}{2}}}_{L^1-L^1} \le 2.
\end{equation} 
The  $L^2$-estimate 
\begin{equation}\label{eq2.4-1}
\norm{\LL_0^{\sfrac{1}{2}} \LL^{-\sfrac{1}{2}}}_{L^2-L^2} \le 1
\end{equation} 
follow readily from the fact that 
\begin{eqnarray*}
\norm{\LL_0^{\sfrac{1}{2}} f}_{L^2}^2 &=& \int_{\RR^d}  | \nabla f |^2 \, dx \\
&\le &  \int_{\RR^d} | \nabla f |^2 \, dx + \int_{\RR^d} V | f |^2 \, dx = \norm{\LL^{\sfrac{1}{2}} f}_{L^2}^2.
\end{eqnarray*}
Hence, by interpolation we have for $p \in [1, 2]$,
\begin{equation}\label{eq2.5}
\norm{\LL_0^{\sfrac{1}{2}} \LL^{-\sfrac{1}{2}}}_{L^p-L^p} \le 2^{\frac{2-p}{p}}.
\end{equation} 
Finally, this estimate together with  the factorisation  $\nabla \LL^{-\sfrac{1}{2}} = \nabla \LL_0^{-\sfrac{1}{2}} \LL_0^{\sfrac{1}{2}} \LL^{-\sfrac{1}{2}}$ implies both  \eqref{eq2.1} and \eqref{eq2.1-1}. 
\end{proof} 

The other type of the Riesz transform $V^{\sfrac{1}{2}} \LL_0^{-\sfrac{1}{2}}$ was studied in \cite{Wrobel}. Using Stein complex interpolation together with the $L^1$-maximal inequality \eqref{eq2.3}, it is shown there that for all $p \in (1,2]$ and $\alpha \in [0, \frac{1}{p}]$,
\begin{equation}\label{eq2.6}
\norm{V^\alpha \LL^{-\alpha}}_{L^p-L^p} \le C_{p,\alpha}
\end{equation}
with a constant $C_{p,\alpha}$  independent of $d$ and of $V$. It is  however not clear  if such  estimate holds  with some  specific constant as in \eqref{eq2.1} when $\alpha = \sfrac{1}{2}$.

%%%%%%%%%%%%%%%%%%%%%%%%%%%%%%%%%%%%%%%%%%%%%%%%%%%%%%%%%%%%
\section{Vector-valued Schr\"odinger operators}\label{sec3}

This section is devoted to Schr\"odinger operators on $L^2(\RR^d, \ell^2)$. We start by the definition of these  operators and then prove Gaussian heat kernels and give  some consequences. The Gaussian bounds will be used in the next section to prove the boundedness of the corresponding  Riesz transforms on $L^p(\RR^d, \ell^2)$. It is worth pointing out that all that we do here for ${\mathcal A} = -{\bf \Delta} + {\bf V}$ on $L^p(\RR^d, \ell^2)$ can be done exactly in the same way for ${\mathcal A}$ on $L^p(\Omega, \ell^2)$ with $\Omega$ an arbitrary domain of $\RR^d$ and the operator is subject to Dirichlet boundary conditions. The Laplacian can also be replaced by a more general elliptic operator with bounded real measurable coefficients. We keep working with $-{\bf \Delta} + {\bf V}$ on $L^p(\RR^d, \ell^2)$ only for simplicity of the exposition and notation. 

Let $d \ge 1$ and let ${\bf V} : \RR^d \to {\mathcal L}(\ell^2)$ be weakly measurable and such that for a.e. $x$, ${\bf V}(x)$ is a symmetric and nonnegative operator on $\ell^2$. The latter property means that   
 for a.e. $x$ and all $\varphi \in \ell^2$, 
 \begin{equation}\label{eq-acc}
 \langle {\bf V}(x) \varphi, \varphi \rangle_{\ell^2} \ge 0.
 \end{equation} 
 We also suppose that ${\bf V}$ is locally integrable in the sense that $\int_{\RR^d} \xi(x) \langle {\bf V}(x) \varphi, \psi \rangle_{\ell^2}\, dx < \infty$ for every $\xi \in C_c^\infty(\RR^d)$ and  $\varphi, \psi \in \ell^2$. 
 
 We define the  form
$$ \fa(u,v) = \int_{\RR^d} \langle \nabla {\bf u}, \nabla {\bf v} \rangle_{\ell^2}\, dx +  \int_{\RR^d} \langle {\bf V} {\bf u},  {\bf v} \rangle_{\ell^2}\, dx = \sum_{j=1}^d \int_{\RR^d} \langle \partial_j {\bf u}, \partial_j {\bf v} \rangle_{\ell^2}\, dx +  \int_{\RR^d} \langle {\bf V} {\bf u},  {\bf v} \rangle_{\ell^2}\, dx $$
with domain
$$D(\fa) = \{ {\bf u} \in W^{1,2}(\RR^d, \ell^2), \  \int_{\RR^d} \langle {\bf V} {\bf u},  {\bf u} \rangle_{\ell^2}\, dx < \infty \}.$$
It is a densely defined, closed and symmetric form.  Hence, it has an associated  operator ${\mathcal A}$ in $L^2(\RR^d, \ell^2)$. This operator is a self-adjoint realization of $-{\bf \Delta} + {\bf V}$. As in the scalar case, we write 
${\mathcal A} = -{\bf \Delta} + {\bf V}$. 
We denote by $(e^{-t\mA})$ the associated semigroup on $L^2(\RR^d, \ell^2)$. This semigroup extends to $L^p(\RR^d, \ell^2)$, $1 \le p < \infty$,  as a strongly continuous and even holomorphic semigroup. This will be obtained as a consequence of the following domination property. Before we state  the result let us denote by $\lambda_{\bf V}(x) = \inf \sigma({\bf V}(x))$ the bottom of the spectrum of ${\bf V}(x)$ as an operator on $\ell^2$. Then of course $\lambda_{\bf V}$ is non-negative and we define as usual the Schr\"odinger operator $H := -\Delta + \lambda_{\bf V}$ on $L^2(\RR^d)$. 
\begin{theorem}\label{thm3.1} Suppose that ${\bf V}$ satisfies the above assumptions. For every $t \ge 0$ and ${\bf f} \in L^2(\RR^d, \ell^2)$ the following pointwise inequalities are  satisfied
\begin{equation}\label{eq3.1}
\norm{e^{-t\mA} {\bf f}(x)}_{\ell^2} \le e^{-tH} \norm{{\bf f}}_{\ell^2}(x) \le  e^{t\Delta} \norm{{\bf f}}_{\ell^2}(x).
\end{equation}
\end{theorem} 
Before we discuss the proof we make some observations. 
One can define a magnetic Schr\"odinger operator $(\nabla -i {\bf a})^*(\nabla -i {\bf a}) + {\bf V}$ as an operator on $L^2(\RR^d, \ell^2)$ and prove the diamagnetic inequality. That is,  the associated semigroup is dominated by the Sch\"odinger semigroup 
$e^{-tH}$ as in \eqref{eq3.1}.  This is proved in \cite{ALLOR} in the setting of  $L^2(\RR^d, \CC^m)$ for any  $m \in \NN$. One can either repeat the arguments from \cite{ALLOR} or use an approximation argument since  $m$ does not play any role. Some other related results to Theorem \ref{thm3.1} are also given in \cite{ALLR} in the setting of operators on $L^2(\RR^d, \CC^m)$. We also point out that the symmetry assumption of ${\bf V}(x)$ can be removed. What is needed is that  $\lambda_{\bf V}(x)$ is a function such that $\Re\, \langle {\bf V}(x) \varphi, \varphi \rangle_{\ell^2} \ge \lambda_{\bf V}(x) \norm{\varphi}_{\ell^2}^2$ for all 
$\varphi \in \ell^2$. The statement and proof of the theorem remain valid in many other settings such as operators on domains or on manifolds. 

\begin{proof}
The second inequality in \eqref{eq3.1} is a very classical result and hence we confine our discussion to the first domination property. 

The proof is based on criteria for domination of semigroups, see \cite{Ouh96}, \cite{Ouh} or \cite{Ouh05}, Theorem~2.30. Denote by $\fb$ the sesquilinear form of the operator $H$, that is
$$ \fb(u,v) = \sum_{j=1}^d \int_{\RR^d} \partial_j u. {\overline \partial_j v} \, dx + \int_{\RR^d} \lambda_{\bf V} u {\overline v}\, dx$$
for $u, v$  in the domain $ D(\fb)$ of $\fb$ which given by 
$$D(\fb) = \{ u \in W^{1,2}(\RR^d), \ \int_{\RR^d} \lambda_{\bf V} | u |^2 \, dx < \infty \}.$$
In order to apply the domination criteria, we have to prove that 
\begin{enumerate}[\rm (a)]
\item
if ${\bf u}\in D(\fa)$, then $\norm{\bf u}_{\ell^2} \in D(\mathfrak{b})$,
\item
if ${\bf u} \in D(\mathfrak{a})$ and $h \in D(\mathfrak{b})$ is such that $|h|\le\norm{\bf u}_{\ell^2} $, then $|h| sign{\bf u}\in D(\mathfrak{a})$ and
\begin{equation}\label{eq3.2}
		\Re\,  \mathfrak{a}({\bf u},|h| sign{\bf u}) \ge \mathfrak{b}(\norm{\bf u}_{\ell^2},|h|).
\end{equation}
\end{enumerate}
We recall that $sign{\bf u}(x) = \frac{{ {\bf u}}(x)}{\norm{\bf u(x)}_{\ell^2}}$ if ${\bf u}(x)  \not = 0$ and $0$ otherwise. 

The proof of these two properties does not differ much from the scalar case, and hence we will skip some details.

Set for $\epsilon > 0$, ${\bf g}_\epsilon  = | h | \frac{ {\bf u}}{\sqrt{ \norm{\bf u(x)}_{\ell^2}^2 + \epsilon}}$. Then 
$$ \partial_j {\bf g}_\epsilon = (\partial_j |h|) \frac{ {\bf u}}{\sqrt{ \norm{\bf u}_{\ell^2}^2 + \epsilon}} +  |h| \frac{ \partial_j {\bf u}}{\sqrt{ \norm{\bf u}_{\ell^2}^2 + \epsilon}} - |h| \frac{ {\bf u}}{ (\norm{\bf u}_{\ell^2}^2 + \epsilon)^{3/2}} \Re\, \langle \partial_j u, u \rangle_{\ell^2}.$$
From this it is easy to see that $\fa({\bf g}_\epsilon, {\bf g}_\epsilon)$ is uniformly bounded with respect to $\epsilon$. Therefore, after extracting a subsequence,  ${\bf g}_\epsilon$ converges weakly in $D(\fa)$ to some ${\bf v}$. Taking the pointwise limit of ${\bf g}_\epsilon$ shows that $|h| sign{\bf u} = {\bf v} \in D(\fa)$. The same idea gives the the expression of the partial derivative
$$ \partial_j (|h| sign{\bf u}) = (\partial_j |h|) sign{\bf u} +  |h| \frac{ \partial_j {\bf u}}{{\norm{\bf u}_{\ell^2}}}\eins_{\{{\bf u} \not= 0\}}  - |h| \frac{ {\bf u}}{ \norm{\bf u}_{\ell^2}^3} \eins_{\{{\bf u} \not= 0\}} \Re\, \langle \partial_j u, u \rangle_{\ell^2}.$$
Using a similar approximation argument, the weak derivative of ${\norm{\bf u}_{\ell^2}}$ for $u \in D(\fa)$ is given by
$$ \partial_j {\norm{\bf u}_{\ell^2}} = \Re\, \langle \partial_j {\bf u}, sign{\bf u} \rangle_{\ell^2}.$$
Putting all this information together, $\Re\,  \mathfrak{a}({\bf u},|h| sign{\bf u})$ is the sum of four terms $T_1, \cdots, T_4$ with
\begin{equation}\label{eq3.3}
T_1 = \Re\, \sum_{j=1}^d \int_{\RR^d}  \langle \partial_j {\bf u}, (\partial_j |h|) sign{\bf u} \rangle_{\ell^2} \, dx = \sum_{j=1}^d \int_{\RR^d} \partial_j {\norm{\bf u}_{\ell^2}} \partial_j |h|\, dx.
\end{equation}
The second term satisfies 
\begin{equation}\label{eq3.4}
T_2 =  \sum_{j=1}^d \int_{\RR^d} \norm{\partial_j{\bf u}}_{\ell^2}^2 \frac{|h|}{\norm{\bf u}_{\ell^2}}\eins_{\{{\bf u} \not= 0\}}\, dx.
\end{equation}
Next,
\begin{equation}\label{eq3.5}
T_3 = -  \sum_{j=1}^d \int_{\RR^d}  (\Re\, \langle \partial_j{\bf u}, sign{\bf u} \rangle_{\ell^2} )^2  \frac{|h|}{\norm{\bf u}_{\ell^2}}\eins_{\{{\bf u} \not= 0\}}\, dx
\end{equation}
and 
%\begin{equation}\label{eq3.6}
\begin{eqnarray}
T_4 &=&  \int_{\RR^d}  \langle {\bf V}{\bf u}, sign{\bf u} |h| \rangle_{\ell^2}\, dx  =  \int_{\RR^d}  \langle {\bf V}{\bf u}, {\bf u} \rangle_{\ell^2} \frac{|h |}{\norm{\bf u}_{\ell^2}} \eins_{\{{\bf u} \not= 0\}}\, dx \nonumber \\
&\ge& \int_{\RR^d} \lambda_{{\bf V}} {\norm{\bf u}_{\ell^2}} | h|\, dx. \label{eq3.6}
%\end{equation}
\end{eqnarray}
The Cauchy-Schwarz inequality (in $\ell^2$) clearly implies that  $T_2 + T_3 \ge 0$. Therefore, the desired inequality in \eqref{eq3.2} follows from this together with \eqref{eq3.3} and \eqref{eq3.5}.
\end{proof} 

The domination property
$$ \norm{e^{-t\mA} {\bf f}}_{\ell^2}(x) \le  e^{t\Delta} \norm{{\bf f}}_{\ell^2}(x)$$
from the previous theorem shows that for all $t > 0$ and $p \in [1, \infty]$, $e^{-t\mA} $ extends to a contraction operator on $L^p(\RR^d, \ell^2)$. It also implies $L^p-L^q$, $1 \le p  \le q \le \infty$,  smoothing properties
\begin{equation}\label{eq3.7}
\norm{e^{-t\mA} {\bf f}}_{L^q(\RR^d, \ell^2)} \le C\, t^{-\frac{d}{2} (\frac{1}{p}-\frac{1}{q})} \norm{{\bf f}}_{L^p(\RR^d, \ell^2)}
\end{equation}
for all $t > 0$ and ${\bf f} \in L^p(\RR^d, \ell^2)$. In particular, the semigroup $e^{-t\mA} $ is given by a kernel ${\bf p}(t,x,y)$ (the heat kernel of ${\mathcal A}$) in the sense that
$$ e^{-t\mA} {\bf f}(x) = \int_{\RR^d}  {\bf p}(t,x,y) {\bf f}(y)\, dy $$ 
and  ${\bf p}(t,x,y)$ defines  a bounded operator on $\ell^2$. See the Appendix.   The above domination provides  the following Gaussian estimate for its operator norm. 

\begin{corollary}\label{cor3.2}
We have for all $t > 0$ and a.e. $x, y \in \RR^d$
\begin{equation}\label{eq3.8}
\norm{{\bf p}(t,x,y)}_{{\mathcal L}(\ell^2)} \le (4 \pi)^{-\frac{d}{2}}\, t^{-\frac{d}{2}} \exp\{ - \frac{| x- y |^2}{4t} \}.
\end{equation}
Here $|x-y|$ denotes the usual Euclidean distance in $\RR^d$. 
\end{corollary}
In the scalar case, an  elementary  argument  allows one to  convert the domination into comparison of the corresponding kernels. In the vector-valued setting this does seem to be immediate. So we give details of proof. 
\begin{proof}
Let us denote by $G(t,x,y) = (4 \pi)^{-\frac{d}{2}}\, t^{-\frac{d}{2}} \exp\{ - \frac{| x- y |^2}{4t} \}$ the Gaussian kernel. Let  $\varphi \in \ell^2$ and $h$ a nonnegative and bounded function  on $\RR^d$. Theorem~\ref{thm3.1} gives for ${\bf f} \in L^1(\RR^d, \ell^2) \cap L^2(\RR^d, \ell^2)$, 
$$ | \langle e^{-t\mA} (h {\bf f})(x), \varphi \rangle_{\ell^2} | \le e^{t \Delta} (h \norm{{\bf f}}_{\ell^2})(x) \norm{\varphi}_{\ell^2}$$
and hence
$$ | \langle \int_{\RR^d}  {\bf p}(t,x,y) {\bf f}(y) h(y)\, dy, \varphi \rangle_{\ell^2} | \le \int_{\RR^d} \norm{\varphi}_{\ell^2} G(t,x,y)  \norm{{\bf f}}_{\ell^2} (y) h(y)\, dy.$$
Then a simple application of Fubini's theorem implies that 
$$ | \int_{\RR^d}  \langle {\bf p}(t,x,y) {\bf f}(y) , \varphi \rangle_{\ell^2} h(y)\, dy | \le \int_{\RR^d} \norm{\varphi}_{\ell^2} G(t,x,y)  \norm{{\bf f}}_{\ell^2} (y) h(y)\, dy.$$
Since this is true for all nonnegative $h$, it follows that 
$$ |\langle {\bf p}(t,x,y) {\bf f}(y) , \varphi \rangle_{\ell^2} | \le \norm{\varphi}_{\ell^2} G(t,x,y)  \norm{{\bf f}}_{\ell^2} (y).$$
This inequality is in the a.e. sense but the null set on which it fails  may depend on $\varphi$. Since $\ell^2$ is separable we apply this to a countable dense set, and then take the union of the corresponding null sets to conclude that  
$$ \norm{{\bf p}(t,x,y) {\bf f}(y)}_{\ell^2} \le G(t,x,y)  \norm{{\bf f}}_{\ell^2} (y).$$
This latter estimate can be converted into operator norm estimate. To do so, let $\psi \in \ell^2$ and $\Omega$ be a bounded domain of $\RR^d$.  Define 
${\bf f} =  \psi \eins_\Omega$. Then ${\bf f} \in L^2(\RR^d, \ell^2)$ and the previous estimate gives
$$ \norm{\eins_\Omega(y) {\bf p}(t,x,y) \psi }_{\ell^2} \le G(t,x,y)  \norm{\psi}_{\ell^2}.$$
Since this holds for all $\psi \in \ell^2$ it follows that 
$$\norm{\eins_\Omega(y){\bf p}(t,x,y)}_{{\mathcal L}(\ell^2)} \le G(t,x,y).$$
This proves the corollary  since $\Omega$ was an arbitrary  domain of $\RR^d$. 
\end{proof}

As in the scalar case, the Gaussian bound of Corollary~\ref{cor3.2} extends to complex time. Recall that since ${\mathcal A}$ is self-adjoint, its semigroup is bounded holomorphic with angle $\frac{\pi}{2}$ on $L^2(\RR^d, \ell^2)$. The kernel of $e^{-z{\mA}}$ 
is ${\bf p}(z,x,y)$ and we have the following estimate. 
\begin{corollary}\label{cor3.3}
There exist  constants $C$, independent of  ${\bf V}$  such that for all  $z \in \CC$ with $\Re\, z > 0$,
\begin{equation}\label{eq3.9}
\norm{{\bf p}(z,x,y)}_{{\mathcal L}(\ell^2)} \le C\, (\Re\, z)^{-\frac{d}{2}} \exp\{ - \, \Re\, \frac{| x- y |^2}{8z} \}.
\end{equation}
\end{corollary}
\begin{proof} Let $\varphi, \psi \in \ell^2$ and consider the function
$$F(z) = \langle {\bf p}(z,x,y) \varphi, \psi \rangle_{\ell^2}.$$
Set $z= t+is$, and observe that by self-adjointness $\norm{e^{-is{\mathcal A}}}_{L^2(\RR^d, \ell^2)} \le 1$. Therefore, using $e^{-z{\mathcal A}} = e^{-\frac{t}{2}{\mathcal A}} e^{-is{\mathcal A}} e^{-\frac{t}{2}{\mathcal A}}$ and 
\eqref{eq3.7} we see that $\norm{e^{-z {\mathcal A}}}_{L^1(\RR^d, \ell^2)-L^\infty(\RR^d, \ell^2)} \le C_1\, (\Re\, z)^{-\frac{d}{2}}$. Hence 
$$| F(z) | \le C_1\, (\Re\, z)^{-\frac{d}{2}} \norm{\varphi}_{\ell^2} \norm{\psi}_{\ell^2}, \quad \Re\, z > 0.$$
The  Gaussian upper bound gives
$$ |F(t)| \le (4 \pi)^{-\frac{d}{2}}\, t^{-\frac{d}{2}} \exp\{ - \frac{| x- y |^2}{4t} \} \norm{\varphi}_{\ell^2} \norm{\psi}_{\ell^2} \quad t > 0.$$
Now we apply \cite{Ouh05}, Lemma 6.18.
\end{proof}
We can trace the constants $C$ in the corollary but this will not be of interest to us here. Also the explicit constant $\frac{1}{8}$ in the exponential decay will not play any role, any other constant $c > 0$ will be enough for our purpose. 

An interesting  consequence of the Gaussian bound for complex time is the holomorphy of the semigroup. More precisely,

\begin{corollary}\label{cor3.4}
The semigroup $(e^{-t {\mA}})$ is bounded holomorphic of angle $\frac{\pi}{2}$ on $L^p(\RR^d, \ell^2)$ for all $p \in [1, \infty)$.
\end{corollary}
The proof is the same as in the scalar case in \cite{Ouh95} or \cite{Ouh05}, Theorem~6.16. It uses the previous corollary, holomorphy on $L^2(\RR^d, \ell^2)$ and Vitali's theorem. We do not repeat the arguments here. 

%%%%%%%%%%%%%%%%%%%%%%%%%%%%%%%%%%%%%%%%%%%%%%%%%%%%%
\section{Riesz transforms of vector-valued Schr\"odinger operators}\label{sec4}

The aim of this section is to prove that the Riesz transform is bounded on $L^p(\RR^d, \ell^2)$ for $p \in (1, 2]$. It is tempting to follow the ideas in Section~\ref{sec2-2} for the scalar case but this is not clear because  the $L^1$-maximal inequality \eqref{eq2.3} is not known in our vector-valued setting for general nonnegative symmetric potentials ${\bf V}$. This inequality was studied in \cite{ALLR} and \cite{ALLOR} in $L^1(\RR^d, \RR^m)$. It is proved in \cite{ALLOR}  that if the smallest and largest eigenvalues $\lambda_{\bf V}$ and $\Lambda_{\bf V}$ of the matrix ${\bf V}$ are proportional, then \eqref{eq2.3} holds  in $L^1(\RR^d, \RR^m)$. In the setting of the present paper, we have by the same method as in \cite{ALLOR},  the maximal inequality in $L^1(\RR^d, \ell^2)$ if the potential ${\bf V}$ satisfies the pointwise inequality 
\begin{equation}\label{eq4.1}
 \lambda_{\bf V}(x) {\norm{\varphi}}_{\ell^2}^2 \le \langle {\bf V}(x) \varphi, \varphi \rangle_{\ell^2} \le \kappa \lambda_{\bf V}(x) \norm{\varphi}_{\ell^2}^2
 \end{equation}
for some constant $\kappa$ and some nonnegative function $\lambda_{\bf V}$. %This condition coincides with the one in \cite{ALLOR} in the setting of $L^1(\RR^d, \RR^m)$. 

For general ${\bf V}$, we  follow a different strategy and we rely on singular integral operators as in \cite{Coulhon-Duong} for the Riesz transform of the Laplace-Beltrami operator on manifolds. The recent paper \cite{HO}, which extends the approach of \cite{Duong-McIntosh}  for singular integral operators to vector-valued settings, will be useful to us here.  

As in the previous section, we consider ${\bf V}: \RR^d \to \ell^2$,   locally integrable, ${\bf V}(x)$ is symmetric and nonnegative. We again denote by ${\mA} = -{\bf \Delta} + {\bf V}$ the Schr\"odinger operator on $L^2(\RR^d, \ell^2)$. Then we have 

\begin{theorem}\label{thm4.1}
Let  $p \in (1, 2]$. Then the Riesz transform $\nabla {\mA}^{-\einhalb}$ is bounded on $L^p(\RR^d, \ell^2)$ and its norm is bounded by a constant independent of ${\bf V}$. If $p = 1$, then $\nabla {\mA}^{-\einhalb}$ is of weak type $(1,1)$ with a weak type bound independent of ${\bf V}$. 
\end{theorem}

\begin{proof}
As mentioned above, the proof follows the ideas in \cite{Coulhon-Duong}. We give the details in order to convince the reader that the $L^p$-norm of $\nabla {\mA}^{-\einhalb}$ can be bounded by a constant independent of ${\bf V}$. This fact is important for the proof of Theorem~\ref{thm1.1}.

Clearly 
$$ \norm{\nabla {\bf f}}_{L^2(\RR^d, \ell^2)}^2 \le {\fa}({\bf f}, {\bf f}) = \norm{{\mA}^{\einhalb} {\bf f}}_{L^2(\RR^d, \ell^2)}^2$$
and notice  that it  implies that the range $R({\mA}^{\einhalb})$ is dense in $L^2(\RR^d, \ell^2)$. It then follows that 
$\partial_j {\mA}^{-\einhalb}$, initially  defined on $R({\mA}^{\einhalb})$, extends  to a contraction operator on $L^2(\RR^d, \ell^2)$.  
We denote by $\partial_j {\mA}^{-\einhalb}$ (for $j = 1, \cdots, d$)  this extension.  For the remainder of the proof, it is enough by interpolation to prove the weak type $(1,1)$ estimate.

We start with the formula
$$ \partial_j  {\mA}^{-\einhalb} = \frac{1}{\sqrt{\pi}} \int_0^\infty \partial_j e^{-t {\mA}}\,  \frac{dt}{\sqrt{t}}.$$
It implies that $\partial_j {\mA}^{-\einhalb}$ is given by a singular kernel
\begin{equation}\label{eq4.2}
 {\bf K}_j(x,y) = \frac{1}{\sqrt{\pi}} \int_0^\infty \partial_{j,x}  {\bf p}(t,x,y)\,  \frac{dt}{\sqrt{t}}
 \end{equation}
where $\partial_{j,x}$ denotes the partial derivative with respect to the variable $x$. For fixed $r > 0$, observe that
\begin{eqnarray*}
\partial_j {\mA}^{-\einhalb} e^{- r^2 {\mA}} &=& \frac{1}{\sqrt{\pi}} \int_0^\infty \partial_j e^{-(t+r^2) {\mA}}\,  \frac{dt}{\sqrt{t}}\\
&=&  \frac{1}{\sqrt{\pi}} \int_{r^2}^\infty \partial_j e^{-t {\mA}}\,  \frac{dt}{\sqrt{t}}\\
 &=& \frac{1}{\sqrt{\pi}} \int_0^\infty \partial_j e^{-t {\mA}}\,\eins_{\{t > r^2\}}\,  \frac{dt}{\sqrt{t-r^2}}.
 \end{eqnarray*}
 Hence, the kernel ${\bf K}^r_j(x,y)$ of $\partial_j {\mA}^{-\einhalb} e^{- r^2 {\mA}}$ is given by 
\begin{equation}\label{eq4.3}
 {\bf K}^r_j(x,y) = \frac{1}{\sqrt{\pi}} \int_0^\infty \partial_{j,x}  {\bf p}(t,x,y)\, \eins_{\{t > r^2\}}\, \frac{dt}{{\sqrt{t-r^2}}}.
 \end{equation}
According to Proposition~\ref{prop:prop2-1} (property~\ref{prop2-1-2}), once we have proved that there exists a constant $W$, independent of $r$, $y$, $u \in \ell^2$ and ${\bf V}$, such that 
\begin{equation}\label{HO-1}
\int_{|x-y| \ge 2 r} \norm{{\bf K}_j(x,y)u - {\bf K}^r_j(x,y)u}_{\ell^2}\, dx \le W\, \norm{u}_{\ell^2}
\end{equation}
we can conclude that  $\partial_j  {\mA}^{-\einhalb}$ is weak type $(1,1)$ and its weak type estimate can be bounded by  $C_1 + C_2 W$ for some constants $C_1, C_2$ independent of ${\bf V}$.\footnote{See the estimates (2.2)-(2.4) in the  proof of Theorem~1.1 in \cite{HO}. The constants $C_1, C_2$ depend only on the constants involved in the Calder\'on-Zygmund decomposition. They are independent of the operator.}
%Once we prove \eqref{HO} with constant independent of ${\bf V}$ we obtain the second assertion of the theorem. 
The $L^p$-estimate in the theorem will then follow by interpolation since the $L^2$-norm of  $\partial_j {\mA}^{-\einhalb}$ is bounded by $1$. Thus, it suffices to prove property \eqref{HO-1}. For this, we use the weighted gradient estimate stated in the lemma below.  \\
Applying \eqref{eq4.2} and \eqref{eq4.3} yields for $\beta > 0$ as in Lemma~\ref{lem-grad-p},
\begin{eqnarray*}
&&  \int_{|x-y| \ge 2 r} \norm{{\bf K}_j(x,y)u - {\bf K}^r_j(x,y)u}_{\ell^2}\, dx \\
&=& \frac{1}{\sqrt{\pi}}  \int_{|x-y| \ge 2 r} \norm{\int_0^\infty \partial_{j,x}{\bf p}(t,x,y)u [ \frac{1}{\sqrt{t}} - \eins_{\{t > r^2\}}\, \frac{1}{{\sqrt{t-r^2}}} ]\, dt}_{\ell^2}\, dx\\
&\le&  \frac{1}{\sqrt{\pi}}  \int_{|x-y| \ge 2 r} \int_0^\infty \norm{\partial_{j,x}{\bf p}(t,x,y)u}_{\ell^2}| \frac{1}{\sqrt{t}} - \eins_{\{t > r^2\}}\, \frac{1}{{\sqrt{t-r^2}}} |\, dt\, dx\\
&=& \frac{1}{\sqrt{\pi}}  \int_{|x-y| \ge 2 r} \int_0^\infty \norm{\partial_{j,x}{\bf p}(t,x,y)u}_{\ell^2}\, e^{\frac{\beta|x-y|^2}{2t}} e^{\frac{-\beta|x-y|^2}{2t}}| \frac{1}{\sqrt{t}} - \eins_{\{t > r^2\}}\, \frac{1}{{\sqrt{t-r^2}}} |\, dt\, dx\\
&\le & \frac{1}{\sqrt{\pi}}  \int_0^\infty e^{-\frac{\beta r^2}{t}}  \int_{\RR^d} \norm{\partial_{j,x}{\bf p}(t,x,y)u}_{\ell^2}\, e^{\frac{\beta|x-y|^2}{2t}} e^{\frac{-\beta|x-y|^2}{4t}}| \frac{1}{\sqrt{t}} - \eins_{\{t > r^2\}}\, \frac{1}{{\sqrt{t-r^2}}} |\, dx\, dt\\
&\le & \frac{1}{\sqrt{\pi}}  \int_0^\infty e^{-\frac{\beta r^2}{t}}  | \frac{1}{\sqrt{t}} - \eins_{\{t > r^2\}}\, \frac{1}{{\sqrt{t-r^2}}} | \left( \int_{\RR^d} \norm{\partial_{j,x}{\bf p}(t,x,y)u}_{\ell^2}^2\, e^{\frac{\beta|x-y|^2}{t}}\, dx\right)^{\einhalb}\\
&& \times \left(  \int_{\RR^d} e^{\frac{-\beta|x-y|^2}{2t}}\, dx \right)^{\einhalb}\, dt\\
&\le& C_0\,  \norm{u}_{\ell^2} \int_0^\infty e^{-\frac{\beta r^2}{t}}  t^{-\einhalb} | \frac{1}{\sqrt{t}} - \eins_{\{t > r^2\}}\, \frac{1}{{\sqrt{t-r^2}}} |\, dt\\
&\le& W\,  \norm{u}_{\ell^2}
\end{eqnarray*}
for some constant $W$ independent of $r$ and ${\bf V}$ since the constants $C$ and $\beta$ from Lemma~\ref{lem-grad-p} are independent of ${\bf V}$. This proves the theorem.
\end{proof}
\begin{lemma}\label{lem-grad-p} Let $\beta \in (0, \frac{1}{6})$. Then there exists a positive constant $C$,  independent of ${\bf V}$,  such that 
\[
\int_{\RR^d} \norm{\partial_{j,x}{\bf p}(t,x,y)u}_{\ell^2}^2\,   e^{\frac{\beta|x-y|^2}{t}}\, dx \le C\, t^{-1 - \frac{d}{2}}\, \norm{u}_{\ell^2}^2
\]
for all $t > 0$, $y \in \RR^d$ and $u \in \ell^2$.
\end{lemma} 
This lemma was first proved in \cite{Grigoryan:JFA} for the Laplace-Beltrami operator on manifolds. It was then used in an efficient way in \cite{Coulhon-Duong} in the proof of the weak type $(1,1)$ of the Riesz transform on manifolds. We adapt the proof of 
\cite{Grigoryan:JFA} to our vector-valued setting and keep in mind that we want to have constants $C$ and $\beta$  independent of ${\bf V}$. We emphasize that we cannot prove a version of this lemma with the uniform norm, that is, with $\norm{\partial_{j,x}{\bf p}(t,x,y)}_{{\mathcal L}(\ell^2)}^2$ instead of $\norm{\partial_{j,x}{\bf p}(t,x,y)u}_{\ell^2}^2$ for every fixed $u$. This is exactly  why we could not  prove  \eqref{HO-1} with the uniform norm and this the reason  to discuss Proposition~\ref{prop:prop2-1} (property~\ref{prop2-1-2}) in Section~\ref{sec1.5}. 

\begin{proof} Let $\beta \in (0, \frac{1}{6})$ and set for a fixed $u \in \ell^2$ and $t > 0$
\[ 
I := \sum_{j=1}^d \int_{\RR^d} \norm{\partial_{j,x}{\bf p}(t,x,y)u}_{\ell^2}^2\,   e^{\frac{\beta|x-y|^2}{t}}\, dx.
\]
Then
\begin{eqnarray*}
I &=& \sum_{j=1}^d  \int_{\RR^d} \langle \partial_{j,x}( {\bf p}(t,x,y)u\,  e^{\frac{\beta|x-y|^2}{t}}), \partial_{j,x} {\bf p}(t,x,y)u \rangle_{\ell^2}\, dx\\
&& - \sum_{j=1}^d  \frac{2 \beta}{t} \int_{\RR^d} |x_j - y_j| e^{\frac{\beta|x-y|^2}{t}}  \langle  {\bf p}(t,x,y)u , \partial_{j,x} {\bf p}(t,x,y)u \rangle_{\ell^2}\, dx\\
 &=& I_1 + I_2.
\end{eqnarray*}
Recall that ${\bf V}$ is non-negative and hence we can dominate the first term by using the bilinear form $\fa$ as follows\footnote{we ignore some simple facts such as ${\bf p}(t,\cdot,y)u\, e^{\frac{\beta|\cdot-y|^2}{t}}  \in D(\fa)$. One can do this by inferring  a sequence $\psi_n$ of $C_c^\infty$ functions which converges pointwise to $1$ and use an approximation argument. See the proof of Theorem~6.19 in \cite{Ouh05}.}
\begin{eqnarray*}
I_1 &\le & \fa( {\bf p}(t,x,y)u, {\bf p}(t,x,y)u\,  e^{\frac{\beta|x-y|^2}{t}})\\
&=&-  \int_{\RR^d} \langle \frac{d}{dt} {\bf p}(t,x,y)u, {\bf p}(t,x,y)u\, e^{\frac{\beta|x-y|^2}{t}} \rangle_{\ell^2}\, dx\\
&\le& \norm{\frac{d}{dt} {\bf p}(t,x,y)u}_{L^2(\RR^d, \ell^2)} \norm{ {\bf p}(t,x,y)u\, e^{\frac{\beta|x-y|^2}{t}}}_{L^2(\RR^d, \ell^2)}.
\end{eqnarray*}
Recall that ${\bf p}(t,x,y)$ satisfies the Gaussian upper bound \eqref{eq3.8} and hence, since $\beta < \frac{1}{4}$, ${\bf p}(t,x,y)u\, e^{\frac{\beta|x-y|^2}{t}}$ still has a Gaussian upper bound and hence
\[
\norm{ {\bf p}(t,\cdot,y)u\, e^{\frac{\beta|\cdot-y|^2}{t}}}_{L^2(\RR^d, \ell^2)} \le C_0\,  t^{-\frac{d}{4}} \norm{u}_{\ell^2}
\]
with a constant $C_0$ independent of ${\bf V}$. On the other hand, %by the Cauchy formula, it follows from Corollary~\ref{cor3.3} that $\frac{d}{dt} {\bf p}(t,x,y)$ also admits a Gaussian upper bound with an additional $t^{-1}$. Therefore
by the semigroup property
\[ 
{\bf p}(t,x,y) = e^{ -\frac{t}{2}{\mathcal A}} {\bf p}( \frac{t}{2},\cdot,y)(x) \quad {\rm and} \ \int_{\RR^d} {\bf p}(t,x,y) {\bf p}(s,x,y)\, dx = {\bf p}(t+s,y,y).
\]
Hence, using again the Gaussian upper bound \eqref{eq3.8}, 
\begin{eqnarray*}
\norm{ \frac{d}{dt}{\bf p}(t,\cdot,y)u\, e^{\frac{\beta|\cdot-y|^2}{t}}}_{L^2(\RR^d, \ell^2)} &=& \norm{ {\mathcal A} e^{ -\frac{t}{2}{\mathcal A}} {\bf p}( \frac{t}{2},\cdot,y) u}_{L^2(\RR^d, \ell^2)}\\
&\le& \frac{C_0'}{t}  \norm{ {\bf p}( \frac{t}{2},\cdot,y) u}_{L^2(\RR^d, \ell^2)}\\
&=&  \frac{C_0'}{t} \sqrt{ \langle {\bf p}( \frac{t}{2},y,y) u, u \rangle_{\ell^2}}\\
&\le&  C_1\,  t^{-\frac{d}{4} -1} \norm{u}_{\ell^2}
\end{eqnarray*}
with a constant $C_1$ independent of ${\bf V}$. 
Therefore,
\begin{equation}\label{eq4.4}
I_1 \le C_0C_1\,  t^{-\frac{d}{2} -1} \norm{u}_{\ell^2}^2.
\end{equation} 
Concerning the second term $I_2$, we first dominate $\frac{|x_j-y_j|}{\sqrt{t}}$ by $e^{\frac{\beta|x-y|^2}{t}}$ (up to a constant) and obtain
\begin{eqnarray*}
I_2 &\le& \frac{C_3}{\sqrt{t}}\,  \sum_{j=1}^d \int_\RR e^{\frac{2\beta|x-y|^2}{t}} \norm{{\bf p}(t,x,y)u}_{\ell^2}\, \norm{\partial_{j,x}{\bf p}(t,x,y)u}_{\ell^2}\, dx\\
&\le& \frac{C_3}{\sqrt{t}}\,  \sum_{j=1}^d \left( \int_\RR e^{\frac{3\beta|x-y|^2}{t}} \norm{{\bf p}(t,x,y)u}_{\ell^2}^2\, dx\right)^{\einhalb} \left( \int_\RR e^{\frac{\beta|x-y|^2}{t}} \norm{\partial_{j,x} {\bf p}(t,x,y)u}_{\ell^2}^2\, dx\right)^{\einhalb}.
\end{eqnarray*}
Using again the Gaussian  upper bound \eqref{eq3.8} and the fact that $\beta < \frac{1}{6}$ we can estimate the last term by 
\[ \frac{C_4}{\sqrt{t}}\,  t^{-\frac{d}{4}} \sqrt{I}\, \norm{u}_{\ell^2}. 
\]
This together with \eqref{eq4.4}, gives
\[
I \le C_0C_1\,  t^{-\frac{d}{2} -1} \norm{u}_{\ell^2}^2 + C_4\,  t^{-\frac{d}{4} - \frac{1}{2}} \sqrt{I}\, \norm{u}_{\ell^2}
\]
which implies the estimate in the lemma.
\end{proof} 

\begin{proof}[Proof of Theorem~\ref{thm1.1}] Let $(V_n)$ be a sequence of non-negative potentials on $\RR^d$ and each $V_n$ is locally integrable. Let  ${\bf V}$ be  the diagonal operator  whose entries are $V_k, k = 0,1, \cdots$. The vector-valued  Schr\"odinger operator ${\mathcal A}$ acts as a diagonal operator so that for ${\bf f} = (f_k)$, ${\mathcal A}{\bf f} =((-\Delta + V_k)f_k)$. By Theorem~\ref{thm4.1}, the Riesz transform of ${\mathcal A}$ is weak type $(1,1)$ with a bound independent of $(V_k)$ and  this shows  the first assertion. \\
By interpolation with the $L^2(\RR^d, \ell^2)$ estimate,  we obtain a  bound on $L^p(\RR^d, \ell^2)$ for $p \in (1,2]$ by a constant  independent of $(V_k)$. This means exactly that the set $\{ \nabla(-\Delta + V)^{-\einhalb}, \ 0 \le V \in L^1_{loc}(\RR^d) \}$ is ${\mathcal R}$-bounded on $L^p(\RR^d)$.\\
Finally,  it is clear from Theorem~\ref{thm2.1} that $\nabla{\mathcal A}^{-\einhalb}$ is bounded on $L^p(\RR^d, \ell^p)$ with a bound independent of $(V_k)$. By interpolation with the  
$L^p(\RR^d, \ell^2)$ case discussed above, we obtain the third  assertion of the theorem.

\end{proof} 

%%%%%%%%%%%%%%%%%%%%%%%%%%%%%%%%%%%%%%%%%%%
\section{Spectral multipliers}\label{sec6}

The aim of this section is to prove Theorem~\ref{thm1.2}. We state and prove a multiplier theorem  for more general potentials ${\bf V}$. Once specified to a diagonal operator, we obtain  Theorem~\ref{thm1.2}. We keep the notation of Section~\ref{sec3}.

\begin{theorem}\label{thm6.1}
Let ${\bf V} : \RR^d \to {\mathcal L}(\ell^2)$ be weakly  measurable, locally integrable,  and such that for a.e. $x$, ${\bf V} (x)$ is  symmetric and nonnegative operator on $\ell^2$. 
Let $m: [0, \infty) \to \CC$ be a bounded function such that
\begin{equation}\label{eq666}
 \sup_{t > 0} \norm{ \phi(\cdot) m(t \cdot) }_{W^{s, \infty}} < \infty
\end{equation}
for some $s > \frac{d}{2}$ and a nontrivial function $\phi \in C_c^\infty(0, \infty)$. Then the operator $m({\mathcal A})$ is of weak type $(1,1)$ and bounded on $L^p(\RR^d, \ell^2)$ for all $p \in (1,2)$. Both the weak type $(1,1)$  and $L^p$-norm can be estimated by a constant independent of ${\bf V}$.
\end{theorem}

The proof  closely follows the proof in  \cite{DOS} (or \cite{Ouh05}, Chapter 7) and it does not depend strongly on the vector-valued setting of the present paper. However, since we want to have constants independent of ${\bf V}$ we need to write out explicitly the relevant arguments. We emphasize that the proof below works for abstract self-adjoint operators ${\mathcal A}$ on $L^2(X, H)$ where $(X, \rho, \mu)$ is a space of homogeneous type and ${\mathcal A}$ has a heat kernel 
${\bf p}(t,x,y)$ which enjoys a Gaussian upper bound
\[
\norm{ {\bf p}(t,x,y)}_{{\mathcal L}(H)}\le \frac{C_0}{V(x, t^{1/\nu})} \exp\{- c_0 \frac{ \rho(x,y)^{\nu/(\nu -1)}}{t^{1/(\nu -1)}} \}
\]
for some $\nu > 1$. Here $V(x, t^{1/\nu}) = \mu(B(x, t^{1/\nu}))$ is  the volume of the ball $B(x, t^{1/\nu})$ and $H$ is a separable Hilbert space. 
We do not write the proof in this abstract setting and restrict  ourselves to the proof of Theorem~\ref{thm1.2} in order to keep control of the constants with respect to  the potential. By an obvious change of the variable we can consider either $m(\mA)$ or $m(\sqrt{\mA})$. 

It is worth mentioning that if  $m: [0, \infty) \to \CC$ satisfies 
\[ \sup_{\lambda > 0} |\lambda^k\, m^{(k)}(\lambda) | < \infty \quad {\rm for} \ k = 1, \cdots, [\frac{d}{2}] +1,
\]
then it satisfies H\"ormander condition \eqref{eq666} for some  $s > \frac{d}{2}$. In the particular case where $m$ is bounded holomorphic in some open sector of $\CC^+$ containing $(0,\infty)$, then by the Cauchy formula, $m$ satisfies \eqref{eq666}. Therefore, $\mA$ has a bounded holomorphic functional calculus of angle $\nu$ on $L^p(\RR^d, \ell^2)$ for any $\nu > 0$ and $p \in (1, \infty)$. 

For the Euclidean Laplacian, H\"ormander's usual condition to guarantee that  $m(-\Delta)$ is bounded on $L^p(\RR^d)$ for all
$p \in (1, \infty)$ is as in \eqref{eq666} but with the $W^{s,2}$ norm instead of $W^{s, \infty}$. In order to have similar condition on $m$ for  Schr\"odinger operators $-\Delta + V$ one needs additional conditions on the potential. We refer to \cite{DOS, COSY, CSY}  and the references therein for this and for sharp spectral multiplier results  for abstract  self-adjoint operators. It is of interest to investigate similar questions for vector-valued Schr\"odinger operators. 

\smallskip
Since $\mA$ is self-adjoint,
\[
\norm{m(\mA)}_{{\LL}(L^2(\RR^d, \ell^2))} \le \norm{m}_{L^\infty[0,\infty)}
\]
and hence by interpolation,  it is enough to prove the weak type $(1,1)$ statement. As in the scalar case, the operator $m(\sqrt{\mA})$ has a kernel $m(\sqrt{\mA})(x,y)$, see \cite{DOS} or \cite{Ouh05}, Chapter 7. 

\begin{lemma}\label{lem6.1}
Let $s > 0$  and suppose that $m$ is supported in $[0, 1]$. Then
\[
I_\psi := \int_{\RR^d} \norm{m(\sqrt{\mA})(x,y) \psi}_{\ell^2}^2 (1 + |x-y|)^s \, dx \le C \norm{m }_{W^{\frac{s}{2} + \frac{d}{2} + 1, 2}}^2 
\norm{\psi}_{\ell^2}^2
\]
for all $\psi \in \ell^2$. The constant $C$ is independent of $m$, $y$ and ${\bf V}$. 
\end{lemma}

\begin{proof} Set $g(\lambda) = m(\sqrt{\lambda}) e^\lambda$. Then 
$$
m(\sqrt{\mA}) =  g(\mA) e^{-\mA}
=  \frac{1}{2 \pi} \int_\RR \hat{g}(\xi) e^{-(1 - i \xi)\mA}\, d \xi.
$$
Hence,
\[ 
m(\sqrt{\mA})(x,y) \psi = \frac{1}{2 \pi} \int_\RR \hat{g}(\xi) {\bf p}(1 -i \xi, x,y) \psi\, d \xi.
\]
Therefore,
$$ \sqrt{I_\psi} \le \frac{1}{2 \pi} \int_\RR |\hat{g}(\xi) | \norm{{\bf p}(1 -i \xi, \cdot,y) \psi}_{L^2(\RR^d, (1 + |x-y|)^s\, dx, \ell^2)}.$$
Next we estimate $\norm{{\bf p}(1 -i \xi, \cdot,y)}_{L^2(\RR^d, (1 + |x-y|)^s\, dx, \ell^2)}$. For this, we apply Corollary~\ref{cor3.3} and we obtain 
\begin{eqnarray*}
\int_{\RR^d} \norm{{\bf p}(1 -i \xi, x,y) \psi}_{ \ell^2}^2 (1 + |x-y|)^s\, dx
&\le& C^2 \int_{\RR^d} e^{-\frac{|x-y|^2}{4 (1+\xi^2)}} (1 + |x-y|)^s\, dx \norm{\psi}_{\ell^2}^2\\
&\le& C'  \int_{\RR^d} e^{-\frac{|x-y|^2}{8(1+\xi^2)}} (1 + \xi^2)^{\frac{s}{2}}\, dx\norm{\psi}_{\ell^2}^2\\
&\le& C'' (1 + \xi^2)^{\frac{s+d}{2}}\norm{\psi}_{\ell^2}^2.
\end{eqnarray*}
The constant $C''$ is independent of $m$ and ${\bf V}$. Inserting this the definition of $I_\psi$ it follows that by the Cauchy-Schwarz inequality
\begin{eqnarray*}
\sqrt{I_\psi} &\le& C_1 \int_{\RR^d} |\hat{g}(\xi) | (1 + \xi^2)^{\frac{s+d}{4} +\frac{1}{2}} (1 + \xi^2)^{-\frac{1}{2}}\, d\xi\,  \norm{\psi}_{\ell^2}\\
&\le& C_2 \left( \int_\RR |\hat{g}(\xi) |^2 (1 + \xi^2)^{\frac{s+d}{2} +1}\, d\xi \right)^{\einhalb} \norm{\psi}_{\ell^2}\\
&\le& C_3 \norm{m }_{W^{\frac{s}{2} + \frac{d}{2} + 1, 2}} 
\norm{\psi}_{\ell^2}.
\end{eqnarray*}
This finishes the proof of the lemma.
\end{proof}
\begin{lemma}\label{lem6.2}
Let $s > 0$  and suppose that $m$ is supported in $[0, 1]$. Then for any $\varepsilon > 0$
\[
 \int_{\RR^d} \norm{m(\sqrt{\mA})(x,y) \psi}_{\ell^2}^2 (1 + |x-y|)^s \, dx \le C_\varepsilon \norm{m }_{W^{\frac{s}{2} + \varepsilon, \infty}}^2 
\norm{\psi}_{\ell^2}^2
\]
for all $\psi \in \ell^2$. The constant $C_\epsilon$ is independent of $m$, $y$ and ${\bf V}$. 
\end{lemma}
The proof is the same  as in \cite{DOS} or \cite{Ouh05}, Lemma~7.17. It starts with 
\begin{eqnarray*}
\norm{ m(\sqrt{\mA})(\cdot,y) \psi}_{L^2(\RR^d)} &=& \norm{ g(\mA){\bf p}(1,(\cdot,y) \psi}_{L^2(\RR^d)}\\
&\le& e \norm{m}_{L^\infty([0,1]} \norm{\psi}_{\ell^2}.
\end{eqnarray*}
Then we use Lemma~\ref{lem6.1} and the Stein-Weiss interpolation as in  \cite{DOS} or \cite{Ouh05}, Lemma~7.17. Note that Stein-Weiss interpolation holds in the vector-valued setting, see \cite{Triebel}, Theorem~1.18.5.

Following again  \cite{DOS} or \cite{Ouh05}, Lemma~7.18, we obtain by a simple scaling argument the following lemma. The potential does not play any role here. 

\begin{lemma}\label{lem6.3}
Let $s > 0$  and suppose that $m$ is supported  in $[0, r]$. Then for any $\varepsilon > 0$
\[
 \int_{\RR^d} \norm{m(\sqrt{\mA})(x,y) \psi}_{\ell^2}^2 (1 + r|x-y|)^s \, dx \le C_\varepsilon r^d \norm{\delta_r m }_{W^{\frac{s}{2} + \varepsilon, \infty}}^2 
\norm{\psi}_{\ell^2}^2
\]
for all $\psi \in \ell^2$. The constant $C_\epsilon$ is independent of $m$, $y$ and ${\bf V}$. Here $(\delta_r m)(\lambda) = m(r \lambda)$. 
\end{lemma}

With these lemmas we can prove the previous theorem.
\begin{proof}[Proof of Theorem~\ref{thm6.1}]
Writing $m(\lambda) = m(\lambda) - m(0) + m(0)$ and the term  $m(0)$ gives $m(0) I_{\ell^2}$ we assume that $m(0) = 0$. Consider a function $\varphi \in C_c^\infty(0, \infty)$ with support contained in $[\frac{1}{4}, 1]$ and such that
$\sum_{n \in \ZZ} \varphi(2^{-n} \lambda) = 1$ for all $ \lambda > 0$. Then
\[
m(\lambda) = \sum_{n \in \ZZ} m(\lambda) \varphi(2^{-n} \lambda) = \sum_{n \in \ZZ} m_n(\lambda)
\]
with $m_n(\lambda) = m(\lambda) \varphi(2^{-n} \lambda)$.  By classical approximation arguments, it is enough to prove that $\sum_{n=-N}^N m_n(\sqrt{\mA})$ is of weak type $(1,1)$ with a bound independent of $N$ (and again of ${\bf V}$). In order to do so, we apply Proposition~\ref{prop:prop2-1}. We prove that for all $\psi \in \ell^2$
\begin{equation}\label{eq6.1}
\int_{| x-y | \ge r} \norm{\sum_{n=-N}^N 
\left( m_n(\sqrt{\mA})(x,y) - e^{-r^2\mA} m_n(\sqrt{\mA})\right) \psi}_{\ell^2}\, dx \le W_1 \norm{\psi}_{\ell^2}
\end{equation}
for some constant $W_1$ independent of $r > 0$, $y \in \RR^d$, and ${\bf V}$.

Set $g_n(\lambda) = m_n(\lambda) (1- e^{-r^2\lambda^2})$ and
\[
I_n(\psi) = \int_{| x-y | \ge r} \norm{g_n(\sqrt{\mA})(x,y) \psi}_{\ell^2}\, dx.
\]
We start with the obvious inequality
\begin{eqnarray*}
I_n(\psi) \le &&\left( \int_{\RR^d}  \norm{g_n(\sqrt{\mA})(x,y) \psi}_{\ell^2}^2(1+ 2^n |x-y|)^{2s}\, dx\right)^\einhalb\\
&& \times  \left(\int_{|x-y| \ge r} (1 + 2^n|x-y|)^{-2s}\, dx \right)^\einhalb.
\end{eqnarray*}
By Lemma~\ref{lem6.3} we have exactly as in \cite{DOS} or \cite{Ouh05}
\begin{eqnarray*}
\int_{\RR^d} \norm{g_n(\sqrt{\mA})(x,y) \psi}_{\ell^2}^2(1+ 2^n |x-y|)^{2s}\, dx
&\le& C_\varepsilon 2^{nd} \norm{\delta_{2^n}g_n}_{W^{s + \varepsilon, \infty}}^2 \norm{\psi}_{\ell^2}^2\\
\vspace{-2cm}&\le& C'_\varepsilon \sup_{t > 0} \norm{\varphi(\cdot) m(t\cdot)}_{W^{s + \varepsilon, \infty}}^2 \norm{\psi}_{\ell^2}^2
\min(1, r2^n) \norm{\psi}_{\ell^2}.
\end{eqnarray*}
On the other hand
\[
\int_{|x-y| \ge r} (1 + 2^n|x-y|)^{-2s}\, dx \le C' 2^{-nd} \min(1, (2^n r)^{d-2s}).
\]
Putting these estimates together yields for $s > \frac{d}{2}$,
\[
\sum_{n =-N}^N I_n(\psi) \le C''_\varepsilon \sup_{t > 0} \norm{\varphi(\cdot) m(t\cdot)}_{W^{s + \varepsilon, \infty}}\norm{\psi}_{\ell^2}
\]
which gives \eqref{eq6.1}. 

\end{proof} 

Now we prove Theorem~\ref{thm1.2}.
\begin{proof}[Proof of Theorem~\ref{thm1.2}] Suppose the function $m$ satisfies the assumption in the theorem. The first two assertions follow directly from Theorem~\ref{thm6.1} in which we take  ${\bf V}$ to be the diagonal operator  whose entries are $V_k$.

To see the third assertion, observe that  for a given nonnegative potential $V$ on $\RR^d$, the $L^p$-norm of $m(-\Delta +V)$ can be estimated by a constant independent of $V$. This  is a special case of Theorem~\ref{thm6.1}. Therefore, for the diagonal potential ${\bf V}$, $m(\mA): L^p(\RR^d, \ell^p) \to L^p(\RR^d, \ell^p)$ is bounded by a constant independent of $(V_k)$. By interpolation with the $L^p(\RR^d, \ell^2)$ estimate given in Theorem~\ref{thm6.1} we obtain the assertion. 
\end{proof}

\begin{proof}[Proof of Corollary~\ref{cor1.3}] We consider again the case where ${\bf V}$ is the diagonal operator with entries $V_k$.  We apply Theorem~\ref{thm6.1} and \cite{DK2}, Corollary~3.5. In the language of \cite{DK2}, $\mA$ admits a  
$W^{s, \infty}$ functional calculus for $s > \frac{d}{2}$. The Sobolev embedding 
$W^{s+ \frac{1}{2} + \varepsilon, 2}$ into $W^{s, \infty}$ shows that $\mA$ has  $W^{\beta, 2}$ functional calculus for $\beta > \frac{d}{2} + \frac{1}{2}$. Then we apply \cite{DK2}, Corollary~3.5 and notice that $\max(\frac{1}{2}, \frac{1}{{\rm type}(L^p(\RR^d, \ell^2))} -
\frac{1}{{\rm cotype}(L^p(\RR^d, \ell^2))}) = \frac{1}{2}$. The implicit constant in \cite{DK2}, Corollary~3.5, is independent of $(V_k)$ since  its dependence on $\mA$ is  mainly through   ${\mathcal R}$-sectoriality estimates.  \end{proof}

%%%%%%%%%%%%
\section{Appendix} \label{App}
 This appendix is dedicated to the Dunford-Pettis representation theorem for vector-valued operators. This is probably  known to experts but we provide the reader with its  statement and proof. 
 \begin{theorem}\label{thm-DP}
 Let $(X, \mu)$ be a $\sigma$-finite measure space and $H$ a separable Hilbert space. Suppose that $T$ is a bounded linear operator from $L^1(X, H)$ into $L^\infty(X,H)$. Then there exists $K: X \times X \to {\mathcal L}(H)$  strongly measurable such that for every ${\bf f} \in L^1(X, H)$
 \[ 
 T{\bf f}(x) = \int_X K(x,y) {\bf f}(y)\, d\mu(y) \quad {\rm a.e.} \, x \in X.
 \]
 In addition, $K$ is bounded and $\operatorname*{ess-sup}_{(x,y) \in X\times X}\norm{K(x,y)}_{{\mathcal L}(H)} = \norm{T}_{L^1(X,H)\to L^\infty(X,H)}$. The kernel $K$ is unique (up to a null set of $X \times X$). 
 \end{theorem}
\begin{proof}
Take an orthonormal basis $(e_k)$ of $H$ and define the operator
\[
T_{ij} g = \langle T(e_i g), e_j \rangle_H, \quad {\rm for}\ g \in L^1(X).
\]
Clearly, 
$$ \norm{T_{ij}g}_{L^\infty(X)} \le \norm{T(e_ig)}_{L^\infty(X,H)} \le \norm{T}_{L^1(X,H)\to L^\infty(X,H)} \norm{g}_{L^1(X)}.$$
By the (scalar) Dunford-Pettis theorem, there exists a measurable bounded function $k_{ij}$ on $X \times X$ such that for every $g \in L^1(X)$
\begin{equation}\label{App1}
 T_{ij} g(x) = \int_X  k_{ij}(x,y) g(y)\, d\mu(y) \quad {\rm a.e.}\ x \in X.
 \end{equation}
 Taking the union over $(i,j)$ of null sets, this equality holds for a.e. $x \in X$ independently of $i, j$. Note also that 
 \begin{equation}\label{App2}
 \norm{k_{ij}}_{L^\infty(X\times X)} = \norm{T_{ij}}_{L^1(X)\to L^\infty(X)} \le \norm{T}_{L^1(X,H)\to L^\infty(X,H)}.
 \end{equation}
We can view $H$ as $\ell^2$ and define $K(x,y)$ as the (infinite) matrix whose entries are $k_{ij}(x,y)$. In order to see that for $(x,y)$ outside a null set $K(x,y)$ defines a bounded operator on $H$ we proceed as follows. Take  $\varphi = \sum_{i=1}^N \alpha_i e_i$ and $\varphi = \sum_{i=1}^M \beta_i e_i$. Then 
\begin{equation}\label{App2.2}
\langle K(x,y) \varphi, \psi\rangle_H = \sum_{i,j} \alpha_i \overline{\beta}_j k_{ij}(x,y)
\end{equation}
is the kernel of the operator $T_{\varphi, \psi}: L^1(X) \to L^\infty(X)$,
\[
T_{\varphi, \psi}g(x) = \langle T(\varphi g)(x),  \psi\rangle_H.
\]
Hence, by \eqref{App2} and again the scalar Dunford-Pettis theorem
\[
| \sum_{i,j} \alpha_i \overline{\beta}_j k_{ij}(x,y) | \le \norm{T}_{L^1(X,H)\to L^\infty(X,H)}\norm{\varphi}_H \norm{\psi}_H.
\]
By density and \eqref{App2.2} this implies that
\[
\norm{K(x,y) \varphi}_H \le \norm{T}_{L^1(X,H)\to L^\infty(X,H)} \norm{\varphi}_H
\]
for all $\varphi$ as above. By density again, $K(x,y)$ is a bounded operator on $H$ with norm
\begin{equation}\label{App2.5}
\norm{K(x,y)}_{{\mathcal L}(H)} \le \norm{T}_{L^1(X,H)\to L^\infty(X,H)} \quad {\rm a.e.} \ (x,y) \in X \times X.
\end{equation}
Note that for  $\varphi = \sum_{i=1}^N \alpha_i e_i$ and
$\varphi = \sum_{i=1}^M \beta_i e_i$, the function 
$$(x,y) \mapsto \langle K(x,y) \varphi, \psi\rangle_H$$
is measurable (as a finite linear combination of measurable functions) and by \eqref{App2.5} and approximation,
this function is measurable for all $\varphi, \psi \in H$. That is, $K(\cdot, \cdot)$ is weakly measurable.  Since $H$ is separable, 
$K(\cdot, \cdot)$ is then strongly measurable by Pettis' theorem. 

Next, for ${\bf f} = \sum_{i=1}^N f_i e_i$ with $f_i \in L^1(X)$ we have
\begin{eqnarray*} 
\langle T{\bf f}(x), e_j \rangle_H &=& \sum_i \langle T{ f_i e_i}(x), e_j \rangle_H
= \sum_i \int_X k_{ij}(x,y) f_i(y) \, d\mu(y)\\
&=& \int_X \langle K(x,y) {\bf f} (x), e_j\rangle_H \, d\mu(y).
\end{eqnarray*}
Hence
\[
 T{\bf f}(x) = \int_X K(x,y) {\bf f}(y)\, d\mu(y) 
 \]
for a.e. $x \in X$. By density and \eqref{App2.5}, this extends to all ${\bf f} \in L^1(X, H)$. The above representation gives
\[ 
|T{\bf f}(x)| \le  \operatorname*{ess-sup}_{(x,y) \in X\times X}\norm{K(x,y)}_{{\mathcal L}(H)} \norm{{\bf f}}_{L^1(X, H)}
\]
and hence the equality 
\[
\operatorname*{ess-sup}_{(x,y) \in X\times X}\norm{K(x,y)}_{{\mathcal L}(H)} = \norm{T}_{L^1(X,H)\to L^\infty(X,H)}
\]
is proved. 

The uniqueness of the kernel can be proved as follows. Suppose that for all  ${\bf f}, {\bf g} \in L^1(X,H)$,  
\[ 
\int_X \int_X \langle K(x,y) {\bf f} (y), {\bf g}(x) \rangle_H\, d\mu(y) d\mu(x) = 0.
\]
Applying this to   ${\bf f} = f e_i$ and ${\bf g} = g e_j$, yields 
\[ 
\int_X \int_X f(y) \overline{g(x)} k_{ij}(x,y) d\mu(y) d\mu(x) = 0 \quad {\rm for\  all}\  f, g \in L^1(X).
\]
This gives $k_{ij} (x,y) = 0$ for a.e. $(x,y)$. Taking again the union over $i, j$ of null sets we obtain $K(x,y) = 0$ for a.e. $(x,y) \in X \times X$.
\end{proof}

%%%%%% Litterature via bibtex
%\bibliographystyle{mybib}
%\bibliography{extrapol}

\providecommand{\bysame}{\leavevmode\hbox to3em{\hrulefill}\thinspace}

\end{document}